\documentclass[reqno]{amsart}     
\usepackage{amsmath,amssymb,amsaddr}    
\usepackage{enumerate}
  
\newtheorem{theorem}                   {Theorem} 
\newtheorem{thm}             [theorem] {Theorem} 
\newtheorem{lemma}           [theorem] {Lemma}  
   
\newtheorem{corollary}       [theorem] {Corollary}   
   
\newtheorem{proposition}     [theorem] {Proposition}  
  
\newtheorem{claim}           [theorem] {Claim}

\newtheorem{conjecture}      [theorem] {Conjecture}

\newtheorem{definition}      [theorem] {Definition}

\theoremstyle{remark}
\newtheorem{remark}          [theorem] {Remark}

\newcommand{\eps}{\varepsilon}
\newcommand{\cH}{\mathcal{H}}

\newcommand{\cC}{\mathcal{C}}
\newcommand{\cJ}{\mathcal{J}}
\newcommand{\NN}{\mathbb{N}}
\newcommand{\RR}{\mathbb{R}}
\newcommand{\EE}{\mathbb{E}}
\newcommand{\PP}{\mathbb{P}}
\newcommand{\Bi}{\mathrm{Bi}}

\newcommand{\cond}{\; \middle\vert \;}

\usepackage{color} 
\newcommand{\macrocolour}{\color{red}}

\newcommand{\lbr}{{\macrocolour \mathcal{T}_*}}                  
                  
\newcommand{\ubr}{{\macrocolour \mathcal{T}^*}}               
\newcommand{\br}{{\macrocolour \mathcal{T}}}

\newcommand{\ubfs}{{\macrocolour \mathrm{BFS}}}            
              
\newcommand{\elbr}{{\macrocolour \eps_*}}                           
\newcommand{\compone}{{\macrocolour C_{J_1}}}

\newcommand{\deggen}[2]{{\macrocolour d_{#1}(\gen{#2})}}

\newcommand{\stoptime}{{\macrocolour i_1}}

\newcommand{\expprob}{{\macrocolour \exp(-\Theta(n^{\delta/2}))}}

\newcommand{\cconst}{{\macrocolour c_0}}					
\newcommand{\cconstl}[1]{{\macrocolour c_{#1}}}

\newcommand{\gensize}[1]{{\macrocolour |\partial({#1})|}}                      
                     
\newcommand{\gen}[1]{{\macrocolour \partial({#1})}}

\newcommand{\lcompconst}{{\macrocolour \lambda}}                      
                      
\newcommand{\death}{{\macrocolour \mathcal{D}}}

\newcommand{\primaryStop}{{\macrocolour \mathcal{E}}}

\newcommand{\contrJumps}[2]{\macrocolour d^{(\text{jp})}_{#1}(\gen{#2})}				
\newcommand{\contrPivots}[2]{\macrocolour d^{(\text{pv})}_{#1}(\gen{#2})}				
\newcommand{\brpiv}{\macrocolour \mathcal{T}^{(\text{pv})}}          
\newcommand{\largetime}{{\macrocolour i_2}}

\renewcommand{\macrocolour}{\color{black}}

\begin{document}

\title[Size of the giant component: a short proof]{The size of the giant component in random hypergraphs: a short proof}
\thanks{The authors are supported by Austrian Science Fund (FWF): P26826 and W1230, the third author is also supported by EPSRC Grant No. EP/N004833/1.}

\author[O.~Cooley, M.~Kang]{Oliver Cooley and Mihyun Kang}
\email{\{cooley,kang\}@math.tugraz.at}
\email{christoph.koch@stats.ox.ac.uk}
\address{Institute of Discrete Mathematics\\ Graz University of Technology\\ Steyrergasse 30, 8010 Graz, Austria}
\author[C.~Koch]{Christoph Koch}
\address{Department of Statistics\\University of Oxford\\ 24--29 St.\ Giles, Oxford OX1 3LB, UK}

\date{\today}

\begin{abstract}
We consider connected components in $k$-uniform hypergraphs for the following notion of connectedness: given integers $k\ge 2$ and $1\le j \le k-1$, two $j$-sets (of vertices) lie in the same $j$-component
if there is a sequence of edges from one to the other such that consecutive edges intersect in at least $j$ vertices.

We prove that certain collections of $j$-sets constructed during a breadth-first search process on
$j$-components in a random $k$-uniform hypergraph are reasonably regularly distributed with high probability.
We use this property to provide a short proof of the asymptotic size of the giant $j$-component shortly after it appears.
\end{abstract}

\maketitle

\section{Introduction and main results}

One of the most well-known results in the field of random graphs is the \emph{phase transition} for the emergence of the giant component. Originally observed by Erd\H{o}s and R\'enyi and strengthened by various researchers over the years, we may state the result as follows. We say that an event holds \emph{with high probability}, often abbreviated to whp, if its probability tends to $1$ as $n$ tends to infinity.

\begin{theorem}[\cite{Bollobas84,ErdRen60,JKLP93,Luczak90}]\label{thm:graphcase}
Let $\eps = \eps(n)>0$ satisfy $\eps \rightarrow 0$ and $\eps^3 n \rightarrow \infty$. 
\begin{enumerate}[(a)]
\item If $p=\frac{1-\eps}{n}$, then whp all components of $G(n,p)$ have $O(\eps^{-2}\log (\eps^3n))$ vertices.
\item If $p=\frac{1+\eps}{n}$, then whp the largest component of $G(n,p)$ has size $(1\pm o(1))2\eps n$, while all other components have $O(\eps^{-2}\log (\eps^3n))$ vertices.
\end{enumerate}
\end{theorem}

Our focus in this paper is an extension of this result to $k$-uniform hypergraphs\footnote{A $k$-uniform hypergraph is an ordered pair formed by a vertex set and an edge set, where each edge contains precisely $k$ distinct vertices.}, for which we need to know what we mean by a connected component in a hypergraph.

Given integers $k\ge 2$ and $1\le j \le k-1$ and a $k$-uniform hypergraph $\cH$, we say that two distinct $j$-sets (of vertices) $J_1,J_2$ are $j$-connected if there is a sequence of edges $E_1,\ldots,E_m$ in $\cH$ such that:
\begin{enumerate}
\item $J_1\subset E_1$ and $J_2\subset E_m$;
\item for each $1\le i\le m-1$, $|E_i\cap E_{i+1}|\ge j$.
\end{enumerate}
In other words, we can ``walk'' from $J_1$ to $J_2$ using edges which consecutively intersect
in at least $j$ vertices. Additionally we say that a $j$-set is $j$-connected to itself.
Then $j$-connectedness is an equivalence relation, and a \emph{$j$-connected component}
(or simply \emph{$j$-component}) is an equivalence class of this relation.
(Equivalently, a $j$-component is a maximal set of pairwise $j$-connected $j$-sets.)
The \emph{size} of a $j$-component is the number of $j$-sets it contains.

This provides a whole family of definitions for connectedness. The case $j=1$, also known as \emph{vertex-connectedness}, is by far the
most studied, but larger $j$, which we refer to as \emph{high-order connectedness}, provides new and richer challenges.

Given integers $n,k,j$ and a real number $p\in [0,1]$, let $\cH^k(n,p)$ denote the random $k$-uniform hypergraph with vertex set $[n]:=\{1,\dots,n\}$ in which each $k$-tuple of vertices forms an edge with probability $p$ independently of one another. Furthermore, we define
$$
\hat{p}_{\mathrm{g}}=\hat{p}_{\mathrm{g}}(n,k,j) := \tfrac{1}{\binom{k}{j}-1}\tfrac{1}{\binom{n}{k-j}}.
$$
The following result is a generalisation of the main result in~\cite{CKKgiant}.  

\begin{thm}\label{thm:main}
Let $1 \le j \le k-1$ and let $\eps=\eps(n)>0$ satisfy $\eps\rightarrow 0$, and $\eps^3 n^j\rightarrow \infty$ 
and $\eps^2 n^{1-\delta} \rightarrow \infty$ 
 for some constant $\delta >0$.
\begin{enumerate}[(a)]
\item \label{thm:main:subcrit} If $p=(1-\eps)\hat{p}_{\mathrm{g}}$, then whp all $j$-components of $\cH^k(n,p)$ have size at most $O(\eps^{-2}\log n)$.
\item \label{thm:main:supercrit} If $p=(1+\eps)\hat{p}_{\mathrm{g}}$, then whp the size of the largest $j$-component of $\cH^k(n,p)$ is $(1\pm o(1))\frac{2\eps}{\binom{k}{j}-1}\binom{n}{j}$, while all other $j$-components have size at most $o(\eps n^j)$.
\end{enumerate}
\end{thm}
Note that for $j=1$ the second condition on $\eps$ becomes $\eps \gg n^{-1/3}$, which is best possible,
while for larger $j$ the condition $\eps \gg n^{-\frac{1-\delta}{2}}$, which is probably not best possible,
takes over. We discuss the critical window in more detail in Section~\ref{sec:concluding}.

A weaker version of this result appeared in~\cite{CKKgiant} as Theorem~2,
where the assumption that $\eps^3n^j,\eps^2n^{1-\delta}\to\infty$ was replaced by the
stronger condition $\eps^3n^{1-2\delta}\to\infty$, for some constant $\delta>0$.
The case $k=2$ and $j=1$ is simply Theorem~\ref{thm:graphcase}.
The case $j=1$ for any $k\ge 2$ was proved by Schmidt-Pruzan and Shamir~\cite{SPS85}.

The proof of Theorem~\ref{thm:main} in \cite{CKKgiant} was based on a short proof of Theorem~\ref{thm:graphcase}
due to Bollob\'as and Riordan~\cite{BR12b}. The idea is to study an exploration process modelling the growth of components and analyse this process based on a branching process approximation. In the setting of hypergraphs substantial challenges arise when adapting this agenda. In order to overcome these obstacles, the proof of Theorem~\ref{thm:main} in~\cite{CKKgiant} required
significant and lengthy technical details.
In this paper we show how much of the technical detail can be avoided and thereby a slightly stronger
result can be obtained. 

Throughout the paper we fix integers $k\ge 2$ and $1\le j \le k-1$.

The main contribution of~\cite{CKKgiant} to the proof of Theorem~\ref{thm:main} was a result stating that
certain collections of $j$-sets are \emph{smooth} in the sense that for any
$0\le \ell \le j-1$, no $\ell$-set is contained in ``too many'' $j$-sets of the collection
(see Corollary~\ref{cor:bddstopgendeg}).
Indeed, Corollary~\ref{cor:bddstopgendeg} followed immediately from a far stronger result in~\cite{CKKgiant}. In this paper we show that we can significantly shorten the proof of Theorem~\ref{thm:main} by avoiding this stronger result and proving Corollary~\ref{cor:bddstopgendeg} more directly.

Let us fix a $j$-set $J_1$ and explore the $j$-component containing $J_1$ via a breadth-first search process $\ubfs=\ubfs(J_1)$.
More precisely, given (arbitrary) orderings $\sigma_j$ and $\sigma_k$ of the $j$-sets and $k$-sets respectively,
in $\ubfs$ we start with $J_1$ being active and all other $j$-sets being neutral.
The first generation consists only of the $j$-set $J_1$.
We produce generation $i+1$ from generation $i$ in the following way. For each $j$-set $J$ of generation $i$ in order according to $\sigma_j$, we query all previously unqueried $k$-sets containing $J$, in order according to $\sigma_k$. If such a $k$-set forms an edge, any neutral $j$-sets within it are added to generation $i+1$ and are called \emph{discovered}.

We denote the $i$-th generation of this process by $\gen{i}$. Note that $\gen{i}$ is a set of $j$-sets, which we may also view as a $j$-uniform hypergraph. Thus for $1\le \ell \le j-1$ and an $\ell$-set $L$, we may define the \emph{degree of $L$ in $\gen{i}$}, denoted $d_L(\gen{i})$, to be the number of $j$-sets of $\gen{i}$ that contain $L$. Our goal is to prove that the degrees $d_L(\gen{i})$ behave ``nicely''.

Next, we fix a constant $\delta$ satisfying $0 < \delta < 1/6$,
and think of it as an arbitrarily small constant -- in general our results
become stronger for smaller $\delta$ (the bounds on the error probabilities
become weaker, but are still exponentially small). Furthermore, we fix a real-valued function $\lcompconst=\lcompconst(n)$ such that
\[
n^{-1/2+\delta/2}, n^{-j/3} \ll \lcompconst \ll \eps \ll 1.
\]

We use $\ubfs$ to grow the component of some $j$-set $J_1$ until at the beginning of some round $i\in\NN$ one of the following three stopping conditions is reached:
 \begin{enumerate}
 \renewcommand{\theenumi}{S\arabic{enumi}}
 \item \label{primaryStopCond1} the component of $J_1$ is fully explored (i.e.\ $\gen{i}=\emptyset$);
 \item \label{primaryStopCond2} the (partial) component $\gen{1}\cup\ldots\cup\gen{i}$ has reached size at least $\lcompconst n^j$;
 \item \label{primaryStopCond3} the $i$-th generation $\gen{i}$ has reached size at least $\lcompconst^2 n^j$.
 \end{enumerate}
 Moreover, we denote the (first) round in which any these stopping conditions is invoked\footnote{This is well-defined since $\ubfs$ always terminates in finite time.} by 
 \begin{equation}\label{eq:stoptime}
 \stoptime=\stoptime(\lcompconst):=\min_{i\in\NN}\{\eqref{primaryStopCond1}\vee\eqref{primaryStopCond2}\vee\eqref{primaryStopCond3}\text{ holds in round } i\}.
 \end{equation}

\begin{theorem}\label{thm:bddgendeg}
For any $0 \le \ell \le j-1$, there exists a constant $C_\ell >0$ such that for all $i \le \stoptime$, with probability $1-\expprob$ we have
\begin{equation}\label{eq:bddgendeg}
\Delta_\ell (\gen{i}) \le C_\ell \left( \frac{\gensize{i}}{n^\ell}+n^\delta \right).
\end{equation}
\end{theorem}

Note that Theorem~\ref{thm:bddgendeg} becomes trivial in the case $j=1$, which is the
main reason why vertex-connectedness is so much easier to handle than high-order connectedness.

In Section~\ref{app:motivation}, we show how Theorem~\ref{thm:bddgendeg} can be applied to prove
the supercritical case of Theorem~\ref{thm:main} in a more concise way than in~\cite{CKKgiant}. The proof of
Theorem~\ref{thm:bddgendeg} forms the main body (Section~\ref{sec:mainproof}) of the paper. The strategy is to split the contribution to the degree of an $\ell$-set
$L$ into two parts, called ``jumps'' and ``pivots''. The contributions made
by each of these parts are bounded separately, in Propositions~\ref{prop:jumps}
and~\ref{prop:pivots} respectively.

\section{Application of Theorem~\ref{thm:bddgendeg}: a simple proof of Theorem~\ref{thm:main}}\label{app:motivation}

In this section we show how Theorem~\ref{thm:bddgendeg} can be used to prove
Theorem~\ref{thm:main}~\eqref{thm:main:supercrit} (i.e.\ the hard case of Theorem~\ref{thm:main})
by following the strategy of~\cite{BR12b} for the graph case: we aim to determine asymptotically
the number $L$ of $j$-sets in \emph{large} $j$-components, i.e. containing at least
$\lambda n^j$ many $j$-sets, using the second moment method. Then it is a simple matter
to show that in fact, almost all of these $j$-sets lie in the same component using a sprinkling argument.
Since the argument already appeared in~\cite{CKKgiant}, we will give only an outline here.

We first explore the $j$-component $C_{J_1}$ of $J_1$ using $\ubfs$ until one of the three stopping conditions
\eqref{primaryStopCond1},~\eqref{primaryStopCond2} and~\eqref{primaryStopCond3} is fulfilled.
We define the \emph{partial components} $C_{J_1}(i):=\gen{1}\cup\ldots\cup\gen{i}$ for any $i$.

We can approximate the search process by a Galton-Watson branching process $\br$ starting with a single vertex
(which represents a $j$-set), and in which the number of children of any vertex $v$ is a random variable $X_v$, where
$$
\frac{X_v}{\binom{k}{j}-1} \sim \Bi \left(\binom{n}{k-j},p\right)
$$
and where the $X_v$ are independent of each other. It is clear that this is an upper coupling
for the search process; the fact that the stopping conditions are invoked before the component
grows too large ensures that it will always be a good approximation whp (see Lemma~\ref{lemma:bdddeg}).

We use this branching process approximation to prove that the probability of the event
$\primaryStop$, that one of stopping conditions~\eqref{primaryStopCond2} and~\eqref{primaryStopCond3} is invoked,
is approximately the same as the survival probability of the associated branching process,
i.e.\ the probability that the branching process survives indefinitely.
Standard branching process techniques show that this probability is approximately
$\tfrac{2\eps}{\binom{k}{j}-1}$.
\begin{claim}\label{claim:survival}
$\PP(\primaryStop)=(1+o(1))\frac{2\eps}{\binom{k}{j}-1}$.
\end{claim}
For completeness, we prove this claim in Appendix~\ref{app:survProb}.
Furthermore, conditional on~\eqref{primaryStopCond3} being invoked, $\compone$
will be large whp. Also if~\eqref{primaryStopCond2} is invoked, then clearly $\compone$ is large.
This already shows that the first moment of $L$, the number of $j$-sets in large components, is
$\tfrac{2\eps}{\binom{k}{j}-1} \binom{n}{j}$.

In order to determine the second moment of $L$, we consider a second exploration process
(of a $j$-component $C_{J_2}$)
starting
from another $j$-set $J_2$ outside $C_{J_1}(\stoptime)$. (The contribution to the
second moment from $j$-sets $J_2$ inside $C_{J_1}(\stoptime)$ is easily shown to be negligible.)
In order to ensure independence, we first delete all the $j$-sets of $C_{J_1}(\stoptime)$
from the hypergraph, so any $k$-set containing such a $j$-set may no longer be queried.
However, since we have not deleted many $j$-sets, $\br$ is still a good approximation
for this search process. Again, the probability that this search process becomes large is
approximately $\tfrac{2\eps}{\binom{k}{j}-1}$.

This almost completes the proof of Theorem~\ref{thm:main}~\eqref{thm:main:supercrit},
but there is one more case to consider,
namely that the search process from $J_2$ remained small, while in fact
the component is large, but we did not see this because we had deleted
some $j$-sets. We need to show that the probability of this event is very small,
and in particular contributes negligibly to the second moment.
We therefore need to know how many queries we have not made because of
deleting $C_{J_1}(\stoptime)$.

The first observation is that since most of the $j$-sets of $C_{J_1}(\stoptime)$
were fully explored, many potential such queries had already been made while
exploring $C_{J_1}$. The only ones we might have missed are the ones containing
a $j$-set of $\gen{\stoptime}$ and $j$-set of $C_{J_2}$.

However, given two $j$-sets $J,J'$, one from each of these components, how many
$k$-sets contain both of them? For $j=1$ (so in particular for graphs) this is
simply $\binom{n-2}{k-2}$, but for $j\ge 2$ the answer is fundamentally dependent
on the size of the intersection $J\cap J'$. We therefore need the following
corollary of Theorem~\ref{thm:bddgendeg}.

\begin{corollary}\label{cor:bddstopgendeg}
There exists a real-valued function $\xi=\xi(n)$ satisfying $\xi/\log n\to\infty$ and $\xi=o(\lcompconst^2n)$ and constants $\kappa_1,\ldots,\kappa_{j-1}>0$ (independent of $n$) such that, conditional on $\primaryStop$, with probability at least $1-\exp\left(-\Theta(\xi^{1/4})\right)$, for all $1\le \ell\le j-1$ and $\ell$-sets $L$ we have
\begin{equation}\label{eq:bddstopgendeg}
d_L(\gen{\stoptime})\le \kappa_\ell \left(|\gen{\stoptime}|n^{-\ell}+\xi\right).
\end{equation}
\end{corollary}

Now given $J \in C_{J_2}$,
we can consider all subsets $L$ of $J$ and bound from above the number of $j$-sets
of $\gen{\stoptime}$ that intersect $J$ in $L$.

Thus we obtain an upper bound on the number of $k$-sets that we have not queried
because of deleting $C_{J_1}(\stoptime)$. With some careful calculation, we
observe that the expected number of edges  among these $k$-sets is small, and
applying Markov's inequality, the probability that there is at least one such
edge is small enough that it is negligible.

It remains to show how Corollary~\ref{cor:bddstopgendeg} can be deduced from Theorem~\ref{thm:bddgendeg}.
\begin{proof}[Proof of Corollary~\ref{cor:bddstopgendeg}]
Note that Theorem~\ref{thm:bddgendeg} is a generalisation of Corollary~\ref{cor:bddstopgendeg} in that it applies to \emph{any} generation, and furthermore, is not conditional on the event $\primaryStop$.

Let us denote by~$\mathcal{A}$ the event that~\eqref{eq:bddstopgendeg}
(i.e. the conclusion of Corollary~\ref{cor:bddstopgendeg}) holds
and by $\mathcal{B}$ the event that~\eqref{eq:bddgendeg} (i.e. the conclusion of Theorem~\ref{thm:bddgendeg})
holds for all $i\le \stoptime$.
Then we have $\mathcal{B}\subset \mathcal{A}$ and therefore
\begin{align*}
\PP(\mathcal{A}|\primaryStop) \ge 1-\frac{\Pr(\overline{\mathcal{B}}\wedge \primaryStop)}{\PP(\primaryStop)}
\ge 1-\frac{\Pr(\overline{\mathcal{B}})}{\PP(\primaryStop)}.
\end{align*}
 Hence, using Theorem~\ref{thm:bddgendeg} and Claim~\ref{claim:survival}, we obtain 
\begin{align*}
\PP(\mathcal{A}|\primaryStop)\ge 1-\frac{\expprob}{\Omega(\eps)}= 1-\expprob,
\end{align*}
where the final equality we used the fact that $\log(1/\eps)=o(n^{\delta/2})$.
\end{proof}

This also completes the proof of Theorem~\ref{thm:main}~\eqref{thm:main:supercrit}.

\begin{remark}
As mentioned earlier, Corollary~\ref{cor:bddstopgendeg} already appeared
as Lemma~4 in~\cite{CKKgiant}. The proof was long and complicated,
and in fact a much stronger result was proved (in which the degrees of $\ell$-sets
are asymptotically determined).
Using Theorem~\ref{thm:bddgendeg} allows us to avoid much of this work.
\end{remark}

\section{Preliminaries}\label{sec:prelim}
We first provide some basic properties of the search process $\ubfs$.
Most of these results were already stated in~\cite{CKKgiant,CoKaPe14} in a similar fashion.
Because the proofs are short, we include them for local reference. 

We define two abstract branching processes, which will form upper and lower couplings on the search process $\ubfs$. For these branching processes, the $j$-sets will be represented by vertices.

Let $\ubr$ denote a Galton-Watson branching process starting with a single vertex, and in which the number of children of any vertex is a random variable $X$, where
$$
\frac{X}{\binom{k}{j}-1} \sim \Bi \left(\binom{n}{k-j},p\right).
$$
It is clear that $\ubr$ forms an upper coupling on $\ubfs$, since from any $j$-set we may query at most $\binom{n}{k-j}$ many $k$-sets, each such $k$-set forms an edge with probability $p$, and if it forms an edge, we discover at most $\binom{k}{j}-1$ new $j$-sets. If we actually discover fewer $j$-sets in $\ubfs$, we can artificially add in some dummy $j$-sets, and equally make some additional dummy queries, to ``fill up'' to $\ubr$. We denote this coupling by $\ubfs \prec \ubr$.

For the lower coupling we define a similar branching process $\lbr$, where in this case the number of children of any vertex is distributed as a random variable $Y$, where
$$
\frac{Y}{\binom{k}{j}-1} \sim \Bi \left((1-\elbr)\binom{n}{k-j},p\right)
$$
with some $\elbr=\elbr(n)>0$ satisfying $\lcompconst\ll \elbr\ll \eps$, say $\elbr:=\sqrt{\lcompconst\eps}$.
Note that $\lbr$ does not always form a lower coupling for $\ubfs$, but early on in the search process,
it is very likely to. More specifically, as long as it is still true that from each $j$-set we make at least
$(1-\elbr)\binom{n}{k-j}$ queries to $k$-sets which contain $\binom{k}{j}-1$ undiscovered $j$-sets,
$\lbr$ forms a lower coupling, which we denote by $\lbr \prec \ubfs$.

We first show that whp $\lbr$ will remain a lower coupling during the period
of the search process that we are interested in. We introduce a new stopping time:
let $\largetime$ be the first round $i$ for which either~\eqref{primaryStopCond1}
or~\eqref{primaryStopCond2} is satisfied. Note that we have $\largetime\ge\stoptime$.

\begin{lemma}\label{lemma:bdddeg}
With probability at least $1-\expprob$ for each $0 \le \ell \le j-1$ and $1\le i\le  \largetime$ we have
$$
\Delta_\ell (C_{J_1}(i)) = O(\lambda n^{j-\ell})= o(\elbr n^{j-\ell}).
$$
In particular, with probability at least $1-\expprob$ for any $1\le i\le \largetime$ we have
$$
\lbr\prec \ubfs.
$$
\end{lemma}

Note that this result has a broadly similar flavour to Corollary~\ref{cor:bddstopgendeg} and Theorem~\ref{thm:bddgendeg}, with the crucial difference that the degree bound applies to
the whole component, not to a single generation. While Lemma~\ref{lemma:bdddeg} has a strong resemblance to Lemma~12 in~\cite{CoKaPe14} and Lemmas~14 and~17 in~\cite{CKKgiant}, it is not immediately apparent that these directly imply Lemma~\ref{lemma:bdddeg}. Since a formal proof of this implication would involve checking tedious technicalities, we instead provide a stand-alone proof for Lemma~\ref{lemma:bdddeg} in Appendix~\ref{app:bdddegProof} for completeness.

We use Lemma~\ref{lemma:bdddeg} twice: firstly, we use it in the proof of Claim~\ref{claim:survival} in Appendix~\ref{app:survProb} (with $i\le \largetime$); secondly, we use it to deduce that after seeing a sufficiently large generation in $\ubfs$, the next one will typically not be smaller.

\begin{lemma}\label{lem:stayBig}
With probability at least $1-\expprob$, for all $1\le i\le \stoptime$ such that $|\gen{i}|\ge n$ we have 
$$
|\gen{i+1}|\ge|\gen{i}|.
$$ 
\end{lemma}
This lemma will be applied in Section~\ref{sec:mainproof} during the proofs of Propositions~\ref{prop:jumps} and~\ref{prop:pivots}. 
\begin{proof}
By the coupling $\lbr\prec\ubfs$ provided by Lemma~\ref{lemma:bdddeg}, conditional on $|\gen{i}|$, we may couple $|\gen{i+1}|$ with a random variable $Z_*$ such that $|\gen{i+1}|\ge \left(\binom{k}{j}-1\right)Z_*$, where 
$$
Z_*\sim \Bi\left((1-\elbr)\binom{n}{k-j}|\gen{i}|,p\right) .
$$
Note that $\EE(Z_*)=(1-\elbr)(1+\eps)\left(\binom{k}{j}-1\right)^{-1}|\gen{i}|$, and thus we obtain
\begin{align*}
\PP\left(|\gen{i+1}|<|\gen{i}|\right)&\le \PP\left(Z_*\le \left(\binom{k}{j}-1\right)^{-1}|\gen{i}|\right)\\
&\le \exp\left(-\frac{((1-\elbr)(1+\eps)-1)^2\left(\binom{k}{j}-1\right)^{-1}\gensize{i}}{2(1-\elbr)(1+\eps)}\right)\\
&\le\exp\left(-\Theta(\eps^2n)\right)\le\expprob,
\end{align*}
where the penultimate inequality follows since $\gensize{i}\ge n$. 
\end{proof}

At various points in the proof of Theorem~\ref{thm:bddgendeg} the following result, which is a direct consequence of a Chernoff bound, will be more convenient than the standard
Chernoff bound formulations.

\begin{lemma}\label{lem:Chernoff}
For any $m=m(n)\in \NN$ and $p=p(n) \in [0,1]$ and constant $a >0$ the following holds. Let $X$ be a random variable distributed as $\Bi (m,p)$. Then with probability at least $1-\expprob$,
$$
X\le (1+a)mp + 2n^\delta.
$$
\end{lemma}

\begin{proof}
We split the proof into two cases. If $mp \le n^\delta$, then $X$ is dominated by a random variable $Y\sim \Bi(m,p')$, where $p':=n^\delta/m$ (if $p' >1$, then $m<n^\delta$ and the upper bound is trivial). Let $\mu:= \mathbb{E}(Y)=n^\delta$. Now using a multiplicative Chernoff bound, which states that
$$
\PP(Y \ge (1+\eta)\mu) \le \left(\frac{e^\eta}{(1+\eta)^{1+\eta}}\right)^\mu
$$
we have
\begin{align*}
\PP (X \ge 2n^\delta) & \le \PP (Y \ge 2n^\delta) = \PP(Y\ge 2\mu) \le \left(\frac{e}{2^2}\right)^{n^\delta} = \exp \left( -\Theta\left(n^\delta\right) \right).
\end{align*}

On the other hand, if $mp>n^{\delta}$, then by an additive Chernoff bound, we obtain
$$
\PP\left(X\ge (1+a)mp\right)\le \exp\left(-\frac{a^2mp}{2(a+1/3)}\right)=\exp\left(-\Omega\left(n^\delta\right)\right), 
$$
and the statement follows.
\end{proof}

\section{Proof of main result}\label{sec:mainproof}

We prove Theorem~\ref{thm:bddgendeg} for a set of recursively defined constants $C_\ell$. (Note that we have made no attempt to optimise these constants.)

We first define $\cconstl{\ell}:= \binom{k-\ell}{j-\ell}-1$ for each $0\le \ell \le j-1$, and set $w_0=w_0(\ell):= \max\{0,j+\ell-k\}$ and
$r_\ell:= \frac{\cconstl{\ell}}{\cconst}$. Note that $r_\ell <1$ for $\ell >0$.
Therefore we can fix some constant $\alpha>0$ such that
$r_\ell' := (1+\alpha)(1+\eps)r_\ell <1$ for all $1\le \ell \le j-1$.

We set
$$
C_0:=1
$$
and iteratively for $1\le \ell \le j-1$ we let
\begin{align}\label{eq:recursion}
C_\ell' & := \binom{k-\ell}{j-\ell} \max \left\lbrace \frac{(1+\alpha)(1+\eps)(k-j)!}{\cconst}
\sum_{w=w_0}^{\ell-1} \binom{\ell}{w}\frac{C_w}{(k-j-\ell+w)!} , 3 \right\rbrace \\ 
C_\ell & := \frac{C_\ell' +2\cconstl{\ell}+1}{1-r_\ell'}. \label{eq:recursion2}
\end{align}
Let us observe that $C_\ell' = f(C_0,\ldots,C_{\ell-1};k,j,\ell)$, while $C_\ell = g(C_\ell';k,j,\ell)$, so these constants are recursively well-defined.

We first note that Theorem~\ref{thm:bddgendeg} is trivial in the case $\ell=0$, since then $\Delta_\ell(\gen{i}) = \gensize{i} = C_0 \frac{\gensize{i}}{n^\ell}$ for any $i$. Therefore in the remainder of the proof we will assume that $\ell \ge 1$.

The proof of Theorem~\ref{thm:bddgendeg} relies on distinguishing two types of contribution to the degrees of $\ell$-sets. 
\begin{definition}\label{def:jumppiv} Let $L$ be an $\ell$-set with $1\le \ell\le j-1$ and let $i\ge 1$.
 \begin{enumerate}[(i)]
 \item A \emph{jump} to $L$ (in step $i$) occurs when we query a $k$-set containing $L$ from a $j$-set in $\gen{i-1}$ which did not contain $L$ and the $k$-set forms an edge of $\cH^k(n,p)$. Such an edge contributes at most $\binom{k-\ell}{j-\ell}$ to $\deggen{L}{i}$.
 \item A \emph{pivot} at $L$ (in step $i$) occurs when we query any $k$-set from a $j$-set in $\gen{i-1}$ containing $L$ and it forms an edge of $\cH^k(n,p)$. Such an edge contributes at most $\binom{k-\ell}{j-\ell}-1$ to $\deggen{L}{i}$. 
 \end{enumerate}
 \end{definition}
 
The following two propositions bound the contribution to the degree $\deggen{L}{i}$ made by jumps and pivots, respectively. We first state these propositions and then show how they combine to prove Theorem~\ref{thm:main}, before providing their proofs.

\begin{proposition}\label{prop:jumps}
For any $0 \le \ell \le j-1$, any $\ell$-set $L$ and any $2\le i \le \stoptime$, with probability at least $1-\expprob$ the contribution made to $\deggen{L}{i}$ by jumps is
$$
\contrJumps{L}{i} \le C_\ell' \left( \frac{\gensize{i}}{n^\ell}+n^\delta \right).
$$
\end{proposition}

\begin{proposition}\label{prop:pivots}
For any $0 \le \ell \le j-1$, any $\ell$-set $L$ and any $2\le i \le \stoptime$, with probability at least $1-\expprob$ the contribution made to $\deggen{L}{i}$ by pivots is
$$
\contrPivots{L}{i} \le (r_\ell'C_\ell + 2c_\ell+1) \left( \frac{\gensize{i}}{n^\ell}+n^\delta \right).
$$
\end{proposition}

\noindent \textbf{Proof of Theorem~\ref{thm:bddgendeg}.}
Now assuming Propositions~\ref{prop:jumps} and~\ref{prop:pivots} hold, we note that
for any $0\le \ell \le j-1$, any $\ell$-set $L$ and any $2\le i \le \stoptime$, assuming the conclusions of Propositions~\ref{prop:jumps}
and~\ref{prop:pivots} hold, we have
$$
\deggen{L}{i} = \contrJumps{L}{i}+ \contrPivots{L}{i} \le (r_\ell' C_\ell + 2\cconstl{\ell} + 1 + C_\ell')\left( \frac{\gensize{i}}{n^\ell}+n^\delta \right)
$$
and furthermore
\begin{align*}
r_\ell' C_\ell + 2\cconstl{\ell} + 1 + C_\ell' = r_\ell' C_\ell + (1-r_\ell')C_\ell \stackrel{\eqref{eq:recursion2}}{=} C_\ell.
\end{align*}
Taking a union bound over all choices of $\ell$, $L$ and $i$ (of which there are certainly at most $jn^{2j}$, and observing that $jn^{2j}\expprob = \expprob$, this completes the proof of Theorem~\ref{thm:bddgendeg}.\qed

\subsection{Jumps: proof of Proposition~\ref{prop:jumps}}

Fix some $\ell$, $L$ and $i$ and let $m_2'=m_2'(L,i)$ denote the number of queries to be made in generation $i$ that would result
in jumps to $L$ if an edge is found. Given a $j$-set $J$, the number of $k$-sets
containing $L$ and $J$ is $0$ if $|J\cap L|<w_0$, and at most $\binom{n}{k-j-\ell+w}$
if $|J\cap L|=w\ge w_0$. Thus we consider the number of $j$-sets in $\gen{i-1}$
which intersect $L$ in $w$ vertices with $w_0\le w\le \ell-1$. This is at most $\binom{\ell}{w} \Delta_w(\gen{i-1})$. Consequently,  we have
\begin{align*}
m_2' & \le \sum_{w=w_0}^{\ell-1} \binom{\ell}{w}\Delta_w(\gen{i-1})\binom{n}{k-j-\ell +w}\\
& \le  \sum_{w=w_0}^{\ell-1} \binom{\ell}{w}C_w\left(\frac{\gensize{i-1}}{n^w}+n^\delta\right)\frac{n^{k-j-\ell+w}}{(k-j-\ell +w)!}=:m_2,
\end{align*}
with probability at least $1-\expprob$, where we have used the induction hypothesis for the second inequality.

The number of such edges we discover is dominated by $\Bi(m_2,p)$, and so by Lemma~\ref{lem:Chernoff}, with probability $1-\expprob$ is at most
$$
(1+\alpha)m_2 p + 2n^\delta.
$$
Furthermore, each such edge results in at most $\binom{k-\ell}{j-\ell}$ new $j$-sets containing $L$ becoming active. Thus we have
\begin{align*}
& \frac{\contrJumps{L}{i}}{\binom{k-\ell}{j-\ell}} \le (1+\alpha)\frac{1+\eps}{\cconst \binom{n}{k-j}} \sum_{w=w_0}^{\ell-1}\binom{\ell}{w}C_w\left(\frac{\gensize{i-1}}{n^w}+n^\delta\right)\frac{n^{k-j-\ell+w}}{(k-j-\ell +w)!} + 2n^\delta\\
& = \frac{(1+\alpha)(1+\eps)(k-j)!}{\cconst(1-O(1/n))} \sum_{w=w_0}^{\ell-1}\binom{\ell}{w}\frac{C_w}{(k-j-\ell+w)!}\left(\frac{\gensize{i-1}}{n^\ell}+n^{-\ell+w+\delta}\right) + 2n^\delta.
\end{align*}
We now observe that since $w\le \ell-1$, the term in $n^{-\ell+w+\delta}$ is always $O(n^{-1+\delta})$, and therefore we may absorb all such terms into the $n^{\delta}$ term by increasing the constant slightly.

We would also like to replace $\gensize{i-1}$ by $\gensize{i}$. This is certainly possible for an upper bound if $\gensize{i-1}\ge n$ by Lemma~\ref{lem:stayBig}. However, if $\gensize{i-1}\le n$, we simply observe that $\frac{\gensize{i-1}}{n^\ell} \le n^{1-\ell}=o(n^\delta)$ (because we have $\ell \ge 1$), and in this case we have $\frac{\contrJumps{L}{i}}{\binom{k-\ell}{j-\ell}} \le 3n^\delta$.

Thus in either case we have
\begin{align*}
\frac{\contrJumps{L}{i}}{\binom{k-\ell}{j-\ell}} & \le \frac{(1+\alpha)(1+\eps)(k-j)!}{\cconst} \sum_{w=w_0}^{\ell-1}\binom{\ell}{w}\frac{C_w}{(k-j-\ell+1)!}\left(\frac{\gensize{i}}{n^\ell}\right) + 3n^\delta\\
& \stackrel{\eqref{eq:recursion}}{\le} \frac{C_\ell'}{\binom{k-\ell}{j-\ell}}\left(\frac{\gensize{i}}{n^\ell}+n^{\delta}\right)
\end{align*}
as claimed.\qed

\subsection{Pivots: proof of Proposition~\ref{prop:pivots}}

Fix some $\ell$, $L$ and $i$ and let $m_3'=m_3'(L,i)$ denote the number of queries to be made in generation $i$ that would result in pivots at $L$ if an edge is found. The number of $j$-sets in $\gen{i-1}$ which could lead to a pivot is at most $\Delta_\ell(\gen{i-1})$ and causes at most $\binom{n}{k-j}$ many queries. Thus, with probability at least $1-\expprob$, we obtain 
\begin{align*}
m_3' & \le \Delta_\ell(\gen{i-1})\binom{n}{k-j}\le C_\ell\left(\frac{\gensize{i-1}}{n^\ell}+n^\delta\right)\binom{n}{k-j}=:m_3,
\end{align*}
  where we have used the induction hypothesis for the second inequality. 

The number of such edges we discover is dominated by $\Bi(m_3,p)$, and so by Lemma~\ref{lem:Chernoff}, with probability $1-\expprob$ is at most
$$
(1+\alpha)m_3 p + 2n^\delta.
$$
Furthermore, each such edge results in at most $\binom{k-\ell}{j-\ell}-1=\cconstl{\ell}$ new $j$-sets containing $L$ becoming active (the one from which we are currently querying is already active). Thus we have
\begin{align*}
\frac{\contrPivots{L}{i}}{\cconstl{\ell}} & \le (1+\alpha)\frac{1+\eps}{\cconst \binom{n}{k-j}} C_\ell\left(\frac{\gensize{i-1}}{n^\ell}+n^\delta\right)\binom{n}{k-j} + 2n^\delta\\
\contrPivots{L}{i} & \le (1+O(1/n))r_\ell'C_\ell \frac{\gensize{i-1}}{n^\ell} + (r_\ell'C_\ell + 2\cconstl{\ell})n^\delta,
\end{align*}
where we recall that $r_\ell' = (1+\alpha)(1+\eps)\cconstl{\ell}/\cconst$.

As before, if $\gensize{i-1}\ge n$, we have $\gensize{i-1}\le \gensize{i}$ by Lemma~\ref{lem:stayBig}, while if $\gensize{i-1}\le n$, we have $\frac{\gensize{i-1}}{n^\ell}=o(n^\delta)$.
In either case we have
\begin{align*}
\contrPivots{L}{i} & \le (1+O(1/n))r_\ell'C_\ell \frac{\gensize{i}}{n^\ell} + (r_\ell'C_\ell + 2\cconstl{\ell}+1)n^\delta\\
& \le \left(r_\ell'C_\ell +2\cconstl{\ell}+1\right) \left(\frac{\gensize{i}}{n^\ell} + n^\delta \right)
\end{align*}
as claimed.\qed

\section{Concluding remarks}\label{sec:concluding}

There are several topics which have been studied extensively for random graphs but remain open questions
in random hypergraphs, particularly for $j\ge 2$.

\subsection{Critical window}
With our improvements on Theorem~\ref{thm:main} compared to~\cite{CKKgiant}, we have come
one step closer to determining the width of the critical window for the emergence of a
unique largest $j$-component in $\cH^k(n,p)$. 
However, the lower bound on $\eps$ given by $\eps^2n^{1-\delta}\to\infty$ is probably still not
best possible for $j\ge 2$. We conjecture the following:
\begin{conjecture}\label{con:criticalWindow}
Theorem~\ref{thm:main} holds for all $\eps=\eps(n)$ satisfying $\eps\to 0$ and $\eps^3n^j\to\infty$.
Furthermore, if $p=(1+cn^{-j/3})\hat{p}_\mathrm{g}$ for some fixed $c\in\RR$, then whp all $j$-components
are of size $O(n^{2j/3})$ and there is more than one $j$-component of $\cH^k(n,p)$ of size $\Theta(n^{2j/3})$.
\end{conjecture} 

Note that if $\eps^3 n^j$ is constant, the bounds on the size of the largest component from the super-critical case
($\Theta (\eps n^j)$) and the sub-critical case ($O(\eps^{-2} \log n)$) match up to the $\log n$ term,
suggesting that we have a smooth transition. In particular, this condition is also sufficient for the
sprinkling argument in Section~\ref{app:motivation} to work.

\subsection{Asymptotic normality of the giant}
In the supercritical case, it would be interesting to determine the asymptotic distribution of the size
of the giant component more precisely, as was done for the case $j=1$ in~\cite{BCOK14,BR12}, where
the size of the giant component was shown to tend to a normal distribution.

\subsection{Nullity}
One can also consider the structure of the components in the subcritical, critical, or supercritical regimes.
For graphs it is well-known that whp all components in the subcritical regime, and whp all except the giant
in the supercritical regime, are either trees or contain at most one cycle. Thus we aim to generalise
the notion of a tree to $j$-connectedness. One possibility is via the nullity (with respect to $j$-connectedness)
of a $k$-uniform hypergraph $\cH=(V,E)$, which we define to be
$$
\nu=\nu(\cH;j,k):=|\cC|+\cconst|E|-\left|\binom{V}{j}\right|,
$$ 
where $\cC$ denotes the set of $j$-components of $\cH$, and we recall that $\cconst=\binom{k}{j}-1$. For a collection $\cJ\subset \binom{V}{j}$ and an edge set $E\subset \binom{V}{k}$ we define the pair $(\cJ,E)$ to be a \emph{(hyper-)tree} if $\cJ$ is a $j$-connected component in $\cH$ and $\nu=0$.\footnote{Note that this is not the only reasonable definition of a tree in hypergraphs.}

In contrast to graphs, a $j$-component does not necessarily contain a spanning tree, but
in~\cite{CKKgiant} it was asked what the minimal nullity of a spanning structure
in the giant $j$-component after the phase transition is. We might also ask about the nullities
of other components in the subcritical and supercritical regimes, and whether these are
small whp, as is the case for graphs.

An approach based on nullities may also help to analyse components within the critical regime,
as was done for graphs in e.g.\ ~\cite{ABBG12,Aldous97,JKLP93,Luczak90}.

\ 

\bibliographystyle{amsplain}
\bibliography{Bibliography}

\providecommand{\bysame}{\leavevmode\hbox to3em{\hrulefill}\thinspace}
\providecommand{\MR}{\relax\ifhmode\unskip\space\fi MR }
\providecommand{\MRhref}[2]{%
  \href{http://www.ams.org/mathscinet-getitem?mr=#1}{#2}
}
\providecommand{\href}[2]{#2}
\begin{thebibliography}{10}

\bibitem{ABBG12}
L.~Addario-Berry, N.~Broutin, and C.~Goldschmidt, \emph{The continuum limit of
  critical random graphs}, Probability Theory and Related Fields \textbf{152}
  (2012), no.~3, 367--406.

\bibitem{Aldous97}
D.~Aldous, \emph{Brownian excursions, critical random graphs and the
  multiplicative coalescent}, The Annals of Probability \textbf{25} (1997),
  no.~2, 812--854.

\bibitem{BCOK14}
M.~Behrisch, A.~Coja-Oghlan, and M.~Kang, \emph{Local limit theorems for the
  giant component of random hypergraphs}, Combin. Probab. Comput. \textbf{23}
  (2014), no.~3, 331--366. \MR{3189416}

\bibitem{Bollobas84}
B.~Bollob{\'a}s, \emph{The evolution of random graphs}, Trans. Amer. Math. Soc.
  \textbf{286} (1984), no.~1, 257--274. \MR{756039 (85k:05090)}

\bibitem{BR12}
B.~Bollob{\'a}s and O.~Riordan, \emph{Asymptotic normality of the size of the
  giant component in a random hypergraph}, Random Structures Algorithms
  \textbf{41} (2012), no.~4, 441--450. \MR{2993129}

\bibitem{BR12b}
\bysame, \emph{A simple branching process approach to the phase transition in
  {${G}\sb {n,p}$}}, Electron. J. Combin. \textbf{19} (2012), no.~4, Paper 21,
  8. \MR{3001658}

\bibitem{CKKgiant}
O.~Cooley, M.~Kang, and C.~Koch, \emph{The size of the giant high-order
  component in random hypergraphs}, Random Structures \& Algorithms (2018),
  doi: 10.1002/rsa.20761.

\bibitem{CoKaPe14}
O.~Cooley, M.~Kang, and Y.~Person, \emph{Giant components in random
  hypergraphs}, accepted for publication in Combinatorics, Probability and
  Computing. arXiv:1412.6366.

\bibitem{ErdRen60}
P.~Erd{\H{o}}s and A.~R{\'e}nyi, \emph{On the evolution of random graphs},
  Bull. Inst. Internat. Statist. \textbf{38} (1961), 343--347. \MR{0148055 (26
  \#5564)}

\bibitem{JKLP93}
S.~Janson, D.~Knuth, T.~{\L}uczak, and B.~Pittel, \emph{The birth of the giant
  component}, Random Struct. \& Alg. \textbf{4} (1993), no.~3, 231--358.
  \MR{1220220 (94h:05070)}

\bibitem{Luczak90}
T.~{\L}uczak, \emph{Component behavior near the critical point of the random
  graph process}, Random Structures Algorithms \textbf{1} (1990), no.~3,
  287--310. \MR{1099794 (92c:05139)}

\bibitem{SPS85}
J.~Schmidt-Pruzan and E.~Shamir, \emph{Component structure in the evolution of
  random hypergraphs}, Combinatorica \textbf{5} (1985), no.~1, 81--94.
  \MR{803242 (86j:05106)}

\end{thebibliography}

\

\appendix

\section{Lower coupling}\label{app:bdddegProof}
Our aim in this section is to prove Lemma~\ref{lemma:bdddeg}.
We first prove that with very high probability
\begin{equation}\label{eq:notTooBig}
|C_{J_1}({\largetime})| \le 3\lambda n^j.
\end{equation}
For by the definition of $\largetime$ we have
$\gensize{\largetime-1}\le|C_{J_1}(\largetime-1)| < \lambda n^j$,
and therefore the number of queries that we make while exploring generation
$\largetime$ is at most
$$
\lambda n^j \binom{n}{k-j}.
$$
Therefore the expected number of edges we discover while exploring generation $\stoptime$
is at most $\lambda n^j\binom{n}{k-j} p =\frac{1+\eps}{\cconst}\lambda n^j$,
and by a Chernoff bound, with
probability at least $1-\expprob$ the number of edges we discover is at most
$\frac{2\lambda n^j}{\cconst}$. Each such edge gives rise to at most $\cconst$ new $j$-sets
and therefore with probability at least $1-\expprob$ we have
$\gensize{\largetime} \le 2\lambda n^j$.
Thus
we obtain $|C_{J_1}(\largetime)|< \lambda n^j + 2\lambda n^j = 3\lambda n^j$, as claimed.

\begin{proof}[Proof of Lemma~\ref{lemma:bdddeg}]
We now prove the first assertion of Lemma~\ref{lemma:bdddeg} by induction on $\ell$. Note that by monotonicity we may assume $i=\largetime$. The base case $\ell=0$ follows immediately from~\eqref{eq:notTooBig},
so assume $\ell \ge 1$ and that $\Delta_{\ell'}(C_{J_1}(\largetime))\le S_{\ell'}\lambda n^{j-i}$
for each $0\le \ell' \le \ell-1$ and some constants $S_0,\ldots,S_{\ell-1}$.

Fix an $\ell$-set $L$.
Let us consider how the degree of $L$ in $C_{J_1}(\largetime)$ might grow.
We bound the contribution from jumps and pivots separately.

We first consider how many queries we may make from a $j$-set $J \in C_{J_1}(\largetime)$
not containing $L$ to a $k$-set containing $L$, i.e.\ the number of queries
which might lead to a jump to $L$. Given $J$, the number of $k$-sets
containing $L$ and $J$ is at most $\binom{n}{k-j-\ell+w}$,
where $w=|J\cap L|$. Thus we consider the number of $j$-sets in $C_{J_1}(\largetime)$
which intersect $L$ in $w$ vertices. This is at most
$\binom{\ell}{w} \Delta_w(C_{J_1}(\largetime))$.
Thus the number of queries which might result in a jump to $L$ is
\begin{align*}
\sum_{w=w_0}^{\ell-1}\binom{\ell}{w}\Delta_w(C_{J_1}(\largetime))\binom{n}{k-j-\ell+w}
& \le \sum_{w=w_0}^{\ell-1}\binom{\ell}{w}S_w\lambda n^{k-\ell}\\
&\le 2^\ell\max_{w=w_0,\ldots,\ell-1}\{S_w\}\lambda n^{k-\ell}.
\end{align*}
By Lemma~\ref{lem:Chernoff} and the fact that $\lcompconst p=\omega(n^{j-k-1+\delta})$, the number of edges
we find in this way is at most $2p\lambda n^{k-\ell}2^{\ell}\max_{w=w_0,\ldots,\ell-1}S_w$
with probability at least $1-\expprob$.
Finally, each such edge contributes
at most $\binom{k-\ell}{j-\ell}\le 2^k$ to the degree of $L$, and so the contribution made by
jumps is at most
$$
2^{k+\ell+1}\max_{w=w_0,\ldots,\ell-1}\{S_w\}\lambda p n^{k-\ell}
\le 2^{k+\ell+1}\max_{w=w_0,\ldots,\ell-1}\{S_w\}\lambda n^{j-\ell}
$$
with probability at least $1-\expprob$.

On the other hand, let us consider the pivots at $L$ as forming a set of
\emph{pivot processes} -- if we discover a $j$-set $J'$ from a $j$-set $J$ via
a pivot at $L$, then $J$ and $J'$ both contain $L$ and are part of the same
pivot process.
Each $j$-set arising from a jump to $L$
(and possibly also $J_1$ if this contains $L$) gives rise to such a pivot process at $L$.
A pivot process is a search process on $j$-sets containing $L$ in a $k$-uniform hypergraph.
By removing $L$ from each of the $j$-sets, it becomes a search process on $(j-\ell)$-sets
in a $(k-\ell)$-uniform hypergraph. We note that the number of children in such a process
is dominated by a random variable $X^{(\text{pv})}$ with 
$$
\frac{X^{(\text{pv})}}{\binom{k-\ell}{j-\ell}-1 } \sim \Bi \left(\binom{n}{k-j},p\right)
$$
and we therefore define $\brpiv$ to be an abstract branching process on vertices
(which represent $j$-sets containing $L$)
in which the number of children of each vertex has this distribution.
The expected number of children is
$$
(1+\eps)\frac{\binom{k-\ell}{j-\ell}-1}{\binom{k}{j}-1}<1,
$$
where the inequality follows since $\ell \ge 1$ and $\eps =o(1)$.
In other words, $\brpiv$ is subcritical. We can then show that
with probability $1-\expprob$, if we start $x\ge n^\delta$ such processes,
then their combined total size is $O(x)$.

More precisely, consider a set of $x\ge n^\delta$ independent copies of
$\brpiv$. If $x-1$ is the maximum of $\binom{k-\ell}{j-\ell}$ times the
number of jumps to $L$ and $n^\delta$,
then the total number of vertices in these $x$ processes dominates the number
of pivots at $L$.

We imagine generating children via a sequence of Bernoulli
queries with success probability $p$, each success giving rise to
$\binom{k-\ell}{j-\ell}-1$ children.
In order for the processes to reach total size $Cx$ for some constant $C$,
we would need at least $\frac{(C-1)x}{\binom{k-\ell}{j-\ell}-1}$ of the first
$Cx\binom{n}{k-j}$ queries to be successful.
But the probability of this can be bounded by a Chernoff bound:
\begin{align*}
\PP\left(\Bi\left(Cx\binom{n}{k-j},p\right)\ge \frac{(C-1)x}{\binom{k-\ell}{j-\ell}-1}\right)
& \le \exp\left(-\frac{x^2\left(\frac{C-1}{\cconstl{\ell}}-Cp\binom{n}{k-j}\right)^2}{2\left(Cx\binom{n}{k-j}p+1/3\right)}\right)\\
& \le \exp\left(-x\frac{\left(\frac{C-1}{\cconstl{\ell}}-\frac{C(1+\eps)}{\cconst}\right)^2}{2\left(\frac{(1+\eps)C}{\cconst}+\frac{1}{3}\right)}\right)\\
& \le \exp\left(-\Theta\left(n^\delta\right)\right),
\end{align*}
where the last line follows for sufficiently large $C$ because $\cconstl{\ell}<\cconst$ and $\eps =o(1)$
and because $x\ge n^\delta$.

Now recall that with probability at least $1-\expprob$ the number of jumps to $L$ is  at most
$
\le 2^{k+\ell+1}\max_{w=w_0,\ldots,\ell-1}\{S_w\}\lambda n^{j-\ell},
$
therefore for sufficiently large $C$, the total contribution to the degree of $L$ made by pivots is at most
$$
C\left( 2^{k+\ell+1}\max_{w=w_0,\ldots,\ell-1}\{S_w\}\lambda n^{j-\ell}+1\right),
$$
since $\lcompconst n^{j-\ell}\ge \lcompconst n=\omega(n^{\delta})$. 

Consequently, the total number of $j$-sets containing $L$ is at most
$$
(C+1)\left( 2^{k+\ell+1}\max_{w=w_0,\ldots,\ell-1}\{S_w\}\lambda n^{j-\ell}+1\right)
\le (C+2)2^{k+\ell+1}\max_{w=w_0,\ldots,\ell-1}\{S_w\}\lambda n^{j-\ell}
$$
with probability at least $1-\expprob$.
We obtain the inductive step by setting $S_\ell:=(C+2)2^{k+\ell+1}\max_{w=w_0,\ldots,\ell-1}\{S_w\}$
and taking a union bound over all $\ell$-sets $L$ (since $\binom{n}{\ell}\expprob = \expprob$).

Finally, we prove that $\lbr$ is a lower coupling. For given a $j$-set $J\in \gen{i}$, we bound the number of $k$-sets $K$
containing another $j$-set $J' \in \compone (i)$, and which therefore would give fewer than $\cconst=\binom{k}{j}-1$ new $j$-sets.
We distinguish cases based on $\ell=|J\cap J'|$ and observe that the number of such $k$-sets is at most
$$
\sum_{\ell=\max\{0,2j-k\}}^{j-1} \Delta_\ell(\compone(i)) \binom{n}{k-2j+\ell} \le \sum_{\ell=0}^{j-1}o(\elbr n^{j-\ell})n^{k-2j+\ell}
 = o(\elbr n^{k-j}).
$$
Thus the number of $k$-sets that can be queried from $J$ and which contain no further discovered $j$-sets is
$$
\binom{n-j}{k-j}-o(\elbr n^{k-j}) = \left(1-O(\tfrac{1}{n})-o(\elbr)\right)\binom{n}{k-j} \ge \left(1-\elbr\right)\binom{n}{k-j}
$$
as required.
\end{proof}

\section{Survival probability}\label{app:survProb}

\begin{proof}[Proof of Claim~\ref{claim:survival}]
Let $\mathcal{A}$ denote the event that $\lbr\prec\ubfs$ for all $i\le \largetime$. We observe that $\PP(\mathcal{A})\ge 1-\expprob$ by Lemma~\ref{lemma:bdddeg}. Thus we have
\begin{align*}
\PP(\primaryStop)\ge \PP(|C_J({\largetime})|\ge \lcompconst n^j)
& \ge \PP(|C_J({\largetime})|\ge \lcompconst n^j \wedge \mathcal{A})\\
&\ge \PP(|\lbr|\ge \lcompconst n^j \wedge \mathcal{A})\\
&\ge \PP(\lbr \text{ survives}) -\PP(\neg\mathcal{A})\\
&\ge \PP(\lbr \text{ survives}) -\expprob.
\end{align*}

On the other hand, note that by some elementary calculation (see Appendix~A in~\cite{CKKgiant}) the following is holds:  if we condition on the process $\ubr$ dying out, we obtain a subcritical Galton-Watson branching process $\tilde{\mathcal{T}}$ where the number of children of each individual is distributed as a random variable $\tilde{X}$, where
$$
\frac{\tilde{X}}{\binom{k}{j}-1}\sim \Bi\left(\binom{n}{k-j},\tilde{p}\right)
$$
for some $\tilde{p}=\tilde{p}(n)$ satisfying $\left(\binom{k}{j}-1\right)\binom{n}{k-j}\tilde{p}=1-\eps\pm o(\eps)$. Furthermore, we have $\EE(|\tilde{\mathcal{T}}|)=(1\pm o(1))\eps^{-1}$ and thus we obtain the upper bound
\begin{align*}
\PP(\primaryStop)&\le \PP(\ubr \text{ survives})+\PP\left(|\ubr|\ge \lcompconst n^j\cond \ubr \text{ dies out}\right)\\
&\le \PP(\ubr \text{ survives}) +(1\pm o(1))(\eps\lcompconst n^j)^{-1},
\end{align*}
by Markov's Inequality. Note that the last term is $o(\eps)$.

We therefore need to calculate the survival probabilities of $\lbr$ and $\ubr$. We treat both cases in parallel by setting $\br:=\lbr$ or $\br:=\ubr$, $\zeta:= 0$ or $\zeta:=\elbr$ and
$\eps':=\eps-\zeta-\eps\zeta = (1-\zeta)(1+\eps)-1$.

It is slightly
more convenient to consider the event $\death$ of the process $\br$ dying out and calculate its probability $\PP(\death)$. The process dies out
if every subprocess starting at a child of the root also dies out. Recall that $\cconst:=\binom{k}{j}-1$.
Because of the
recursive nature of the tree, we have
\begin{align*}
\PP(\death) & =\sum_{i=0}^\infty \PP\left(\Bi\left((1-\zeta)\binom{n}{k-j},p\right)=i\right)\PP(\death)^{\left(\binom{k}{j}-1\right)i}\\
 & =\sum_{i=0}^\infty \binom{(1-\zeta)\binom{n}{k-j}}{i}\PP(\death)^{i\cconst}(1-p)^{(1-\zeta)\binom{n}{k-j}-i}\\
 &=\left(p\PP(\death)^\cconst+1-p\right)^{(1-\zeta)\binom{n}{k-j}}=\left(1-p(1-\PP(\death)^{\cconst})\right)^{(1-\zeta)\binom{n}{k-j}}.
\end{align*}
We set $x:=1-\PP(\death)^{\cconst}$ and $y:=\left((1-\zeta) \cconst\binom{n}{k-j}\right)^{-1}$ and note that $p=(1+\eps)\hat{p}_{\mathrm{g}}=(1+\eps)\cconst^{-1}\binom{n}{k-j}^{-1}=(1+\eps')y$. Hence, we obtain
\begin{align*}
1-x=\left(1-(1+\eps')yx\right)^{1/y},
\end{align*}
and furthermore solving for $\eps'$ yields
\begin{align*}
\eps'=\frac{1-xy-(1-x)^y}{xy}=\frac{ \frac{y(1-y)}{2}x^2 + \frac{y(1-y)(2-y)}{6}x^3 + \ldots}{xy}
\end{align*}
implying 
\begin{align*}
\eps'=\frac{x}{2}+O(x^2).
\end{align*}
In other words, since $\eps'=(1\pm o(1))\eps$, we have 
\begin{align*}
1-\PP(\death)=1-(1-x)^{1/\cconst}=(1\pm o(1))\frac{2\eps}{\cconst},
\end{align*}
as claimed.
\end{proof}

\end{document}